\documentclass[12pt,a4paper]{amsart}
\usepackage[english]{babel}
\usepackage[utf8]{inputenc}
\usepackage{tikz}
\usepackage{lipsum}
\usepackage{amssymb}
\usepackage{float}
\newtheorem{Theorem}{Theorem}[section]
\newtheorem{Example}[Theorem]{Example}
\newtheorem{Lemma}[Theorem]{Lemma}
\newtheorem{Definition}[Theorem]{Definition}
\newtheorem{Corollary}[Theorem]{Corollary}

\newcommand{\Bg}[2]{\mathrm{Bg}_{#2}(#1)}

\title[Schubert varieties indexed by involutions]{Smoothness of Schubert varieties indexed by involutions in finite simply laced types}

\author{Axel Hultman \and Vincent Umutabazi} 

\address{Department of Mathematics, Link\"oping University, SE-581 83 Link\"oping, Sweden}

\email{axel.hultman@liu.se}
\email{vincent.umutabazi@liu.se}

\begin{document}

\begin{abstract}
We prove that in finite, simply laced types, every Schubert variety indexed 
by an involution which is not the  
longest element of some parabolic subgroup is singular.
\end{abstract}

\maketitle

\section{Introduction} \label{sec:Introduction}
Let $ w $ be an involution in the symmetric group $S_{n} $. In \cite{MR2184447} Hohlweg proved that the Schubert variety $X_w$ is smooth if and only if $ w $ is the longest element of some parabolic subgroup of $ S_{n} $. He arrived at this result by exploiting Lakshmibai and Sandhya's \cite{Vs} classical pattern avoidance criterion for smoothness of type $A$ Schubert varieties.
 
 The main result of this paper extends Hohlweg's result to arbitrary finite, 
 simply laced types. Namely, if $ W $ is a simply laced Weyl group, 
 and $ w\in W $ is an involution, $ X_{w} $ is smooth if and only if $ w $ 
 is the longest element of a parabolic subgroup of $ W $.

It seems likely that it would be possible to arrive at this result in a case by case fashion using the general root system pattern avoidance criteria for smoothness pioneered by Billey and Postnikov \cite{BP}. Instead of investigating this approach, we provide a uniform proof based on Carrell-Peterson type criteria in terms of the associated Bruhat graphs.

In Section \ref{sec:prel}, we recall properties of Coxeter systems and 
Bruhat graphs which can be used to study smoothness of Schubert varieties in a combinatorial way. In Section \ref{sec:main}, we prove Theorem \ref{th:main} which is the main result.

\section{Preliminaries}\label{sec:prel}
In this section some properties of Bruhat graphs of Coxeter groups 
are recalled.  For more on these concepts, see e.g.\ \cite{MR2133266}  and 
\cite{MR1104786}.

A {\em Coxeter group} is a group $W$ generated by a set $ S $ of {\em simple reflections} $ s $ under relations of the 
form $ s^{2}=e $ and $ (ss')^{m(s,s')}=e $ for all $ s,s'\in S $ where $ e $ 
is the identity element and  $m(s',s)= m(s,s')$  $ \geq 2$ is the order of $ ss' $ 
for $ s\neq s' $. 
The pair $ (W,S) $ is called a {\em Coxeter system}, and each element $ w\in W $ is a 
product of generators $s_{i}\in S $, i.e., $ w=s_{1}s_{2}\cdots s_{j} $. 
If $ j $ is minimal among all such 
expressions 
for $ w $, then $ j $ is called the {\em length} of $ w $, denoted $ \ell(w)=j $.  
The Coxeter system $ (W,S) $ is {\em simply laced} if $ m(s, s')\leq 3 $ 
for all  $ s, s' \in S $; otherwise it is {\em multiply laced}.

If $W$ is finite, there exists a {\em longest element} $w_0\in W$. It is an involution and satisfies $\ell(v) < \ell(w_0)$ for all other elements $v\in W$. In fact $w_0$ is the unique element in $w$ such that $\ell(sw_0) < \ell(w_0)$ for all $s\in S$.

From now on let us fix a Coxeter system $(W,S)$. Let  
$ T=\lbrace wsw^{-1}:w\in W, s\in S\rbrace $ 
be the set of {\em reflections} in $ W $. 
For $ v,w \in W $  define: 
\begin{itemize}
\item [ (i) ]  $ v \rightarrow w $ if $ w=vt $ for some $ t\in T $ 
with $ \ell(v)<\ell(w) $. 
\item[ (ii) ] $ v\leq w $ if $ v=v_{0}\rightarrow v_{1}\rightarrow 
\cdots \rightarrow v_{m}=w $ for some $ v_{i}\in W $.
\end{itemize}
The {\em Bruhat graph} $ \Bg{W}{S} $ of $ (W,S) $ is the directed graph whose vertex set
is $ W $ and whose edge set is $ E_{_{\Bg{W}{S}}}=\lbrace (u,w):u\rightarrow w \rbrace $. The {\em Bruhat order} is the partial order relation on $ W $ given by (ii). 
\begin{Example}
	Denote by $ W(A_{2}) $ the Coxeter group of type $ A_{2} $ with set of simple reflections $S(A_2) = \{s_1,s_2\}$ satisfying $m(s_1,s_2)=3$. Then $\Bg{W(A_{2})}{S(A_2)} $ is as depicted in Figure \ref{fi:A2}.
	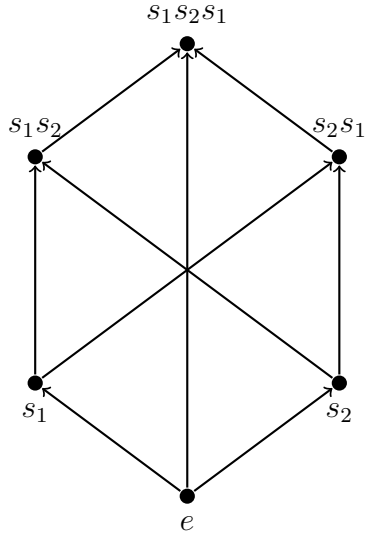
\begin{figure}[H]
		\centering 
		\begin{tikzpicture}[
		thick,
		acteur/.style={
			circle,
			fill=black,
			thick,
			inner sep=0.5pt,
			minimum size=0.2cm
		}
		] 
		\node (a1) at (2,4.5) [acteur,label=$s_{2}s_{1} $]{};
		\node (a2) at (2,1.5)[acteur,label=below:$s_{2} $]{}; 
		\node (a4) at (0,0) [acteur,label=below:$ e $]{}; 
		\node (a5) at (-2,1.5) [acteur,label=below:$s_{1} $]{}; 
		\node (a6) at (-2,4.5) [acteur,label=$s_{1}s_{2} $]{}; 
		\node (a9) at (0,6) [acteur,label=$ s_{1}s_{2}s_{1} $]{}; 
		\draw[<-] (a1) -- (a2);  
		\draw[->] (a5) -- (a1);
		\draw[<-] (a2) -- (a4);
		\draw[<-] (a5) -- (a4);
		\draw[->] (a6) -- (a9);
		\draw[->] (a1) -- (a9);
		\draw[->] (a5) -- (a6);
		\draw[<-] (a9) -- (a4);
		\draw[<-] (a6) -- (a2);
		\end{tikzpicture}
		\caption{The Bruhat graph of $ (W(A_{2}),S(A_2))$.} 
		\label{fi:A2} 
	\end{figure}
\end{Example}
The map $ v\mapsto v^{-1} $ is an automorphism of the Bruhat order: 
\begin{Lemma}\label{le:invert}
	For all $ v,w\in W $, $ v< w $ if and only 
	if $ v^{-1}< w^{-1} $. 
\end{Lemma}
Define the {\em left descent set} of $ w $ as $ D_{L}(w)=\lbrace s\in S: 
\ell (sw)<\ell(w)\rbrace $. The following fundamental result about the Bruhat order is sometimes called the {\em Lifting property}.
\begin{Lemma}[Verma \cite{verma}] \label{le:lifting}
	Suppose $ v<w $ and $ s\in D_{L}(w)\setminus D_{L}(v) $. 
	Then, $ v\leq sw  $ and $ sv\leq w $.
\end{Lemma}
\subsection{Reflection subgroups}
Maintain the Coxeter system $ (W,S) $ and its set of reflections $T$ as defined above. Then $ W' $ is a {\em reflection subgroup} of $W$ if $ W'=\langle W'\cap T\rangle $. A reflection subgroup $ W' $ is called {\em dihedral} if $ W'=\langle t,t' \rangle  $ for some $ t,t'\in T $, with $ t\neq t' $.

\begin{Lemma}[Dyer \cite{MR1104786}]\label{le:dihedral}
	Suppose $ t_{1}, t_{2}, t_{3}, t_{4}\in T  $ and $ t_{1}t_{2}=t_{3}t_{4}\neq e $. 
	Then $ W'=\langle t_{1}, t_{2},  t_{3},  t_{4} \rangle $ is a dihedral reflection 
	subgroup of $W$.
\end{Lemma}
It turns out that reflection subgroups of  $ W $ are themselves Coxeter groups.
For $w\in W$, define $ N(w):=\lbrace t\in T: \ell(tw)<\ell(w)\rbrace $. This is the set of {\em inversions} of $w$.
\begin{Theorem}[Deodhar \cite{deodhar}, Dyer \cite{MR1076077}]\label{th:refsubgroup}
	Let $ W' $ be a reflection subgroup of $W$ and define
	$ X=\lbrace t\in T: N(t)\cap W'=\lbrace t \rbrace \rbrace $. Then,
	\begin{itemize}
		\item[(1)] $ W'\cap T= \lbrace ut'u^{-1}:u\in W', t'\in X \rbrace $.
		\item[(2)] $ (W',X) $ is a Coxeter system.
	\end{itemize}
\end{Theorem}

Coxeter described all types of affine groups generated by 
reflections and their reflection subgroups \cite{MR1503182}. 
The following lemma is a very special 
case. It can be seen directly e.g.\ by considering root lengths.

\begin{Lemma}\label{le:simplylaced}
	Every reflection subgroup of a finite simply laced group is itself simply laced.
\end{Lemma}

For any subset $ Y \subseteq W$ define 
the Bruhat graph of $ Y $,  denoted $ \Bg{Y}{S} $, as the directed subgraph of 
$ \Bg{W}{S}$ induced by $ Y $. 

\begin{Theorem}[Dyer \cite{MR1104786}]\label{th:bgisom}
Let $ W' $ be a reflection subgroup of $W$ and let $ X $ be as in Theorem \ref{th:refsubgroup}. Then $\Bg{W'}{S} = \Bg{W'}{X}$. 
\end{Theorem}
\subsection{Schubert varieties}
Let $ G $ be an algebraic group over $ \mathbb{C} $ and $ B $ a Borel subgroup containing a maximal torus $T$. Then $ G /B $ is called the {\em flag variety} and it is the disjoint union of 
{\em Schubert cells} $ BwB/B $ where $ w\in W $ and $ W = N(T)/T $ is the {\em Weyl group} (which is a finite Coxeter group). 
The closure 
$ X_{w}:=\overline{BwB/B} $ is called a {\em Schubert variety}. 
Note that $ G/B=X_{w_{0}} $
 for $ w_{0} $ the longest element 
of $ W $. More on Schubert varieties can be found e.g.\ in \cite{MR1782635}.

Next, we review ways to detect singularities of Schubert varieties by inspecting Bruhat graphs. For a lower interval $ [e,w]=\{z\in W:e\leq z \leq w\} $ write $ \Bg{w}{S} $ for the Bruhat graph of $ [e,w] $. Let $ z $ be a vertex in $\Bg{w}{S}$. The \emph {degree} of $z$, denoted $ \deg_w(z) $, is the number of edges incident to $ z $ in $\Bg{w}{S}$ (where directions of edges are ignored).

The following result holds in any Coxeter group. In that generality it is due to Dyer \cite{MR1202136}. In our context, where $W$ is a (finite) Weyl group, other proofs given by Carrell and Peterson \cite{MR1278700} and Polo \cite{MR1307965} also apply.

\begin{Theorem}\label{th:degree}
Let $ w\in W $. Then the degree of any vertex in $ \Bg{w}{S} $ is at least $\ell(w)$.
\end{Theorem}
In any Bruhat graph $ \Bg{w}{S} $, it is known that $\ell(w)=|N(w)| = \deg_w(w) $.
In particular, if $ \Bg{w}{S} $ is regular (i.e., every vertex of $ \Bg{w}{S} $ has the same number of edges), then $ \deg_w(e)=\ell(w)$.

\begin{Theorem}[Carrell-Peterson \cite{MR1278700}]\label{th:CP}
	The Schubert variety $X_{w}  $ is rationally smooth 
	if and only if $ \Bg{w}{S} $ is regular.
\end{Theorem}
Smoothness and rational smoothness are equivalent for simply laced Weyl groups:
\begin{Theorem}[Carrell-Kuttler \cite{MR1953262}]\label{th:simplylaced}
	Suppose $ W $ is simply laced. Then for any $ w \in W$,  
	$ X_{w} $ is smooth if and only if it is rationally smooth. 
\end{Theorem} 

\begin{Corollary}\label{co:simplylaced}
If $ W $ is simply laced then 
	$ X_{w} $ is smooth if and only if 
	$ \Bg{w}{S} $ is regular.
\end{Corollary}

In general smoothness is stronger than rational 
smoothness when $ W $ is not simply laced. For example 
$ X_{s_{1}s_{2}s_{1}} $ is rationally smooth but not 
smooth if $ W $ is of type $ C_{2} $ generated by the simple reflections $s_1$ and $s_2$ with $ s_1 $ corresponding to the short root.

The next definition provides another characterization which can be used to 
prove that a given Schubert variety is not rationally smooth (see Theorem \ref{th:brokenrhombi} below).
\begin{Definition}\label{de:brokenrhombi}
\emph{\cite{MR2802176}}
	Let $ x,u,v\leq w $. The Bruhat interval $ [e,w] $ contains the {\em broken rhombus} 
	$ (x,u,v) $ if 
	the conditions below are satisfied: 
	\begin{itemize}
		\item[(1)] $ x\leftarrow u\rightarrow v $;
		\item[(2)]There is some $ y\in W $  with $ x\rightarrow y\leftarrow  v $;
		\item[(3)] If $ x\rightarrow y\leftarrow v $, then $ y\nleq w $.
	\end{itemize}
\end{Definition}
\begin{Example} \emph{Consider the group $W(D_4)$ of type $D_4$ with set of simple reflections $S(D_4) = \{s_1,s_2,s_3,s_4\}$ where $ m(s_{i},s_{2})=3 $ for $ i=1, 3, 4 $ and 
	$ m(s_{i}, s_{j})=2 $ for  $ i, j\neq 2 $. In Figure \ref{fi:D4} is depicted $ \Bg{w}{S(D_4)} $ 
	for $ w=s_{2}s_{1}s_{3}s_{4}s_{2} $. We use  $ 1,2,3,4 $ for $ s_{1},s_{2},s_{3},s_{4} $ respectively for brevity. 
	Thus, for example, $w$ is represented by $ 21342 $.
	 The interval $ [e, w] $  
	 contains the broken rhombus $ (s_{2}s_{3} , s_{2}, s_{1}s_{2}) $ 
	since there is no $ y \le w $  
	such that $ s_{2}s_{3}\rightarrow y\leftarrow  s_{1}s_{2} $  although   
	$ s_{2}s_{3}\rightarrow s_{1}s_{2}s_{3}\leftarrow  s_{1}s_{2} $. Moreover note that $ \Bg{w}{S(D_4)} $ is not regular. Hence, $ X_{w} $ is not smooth for $ w=s_{2}s_{1}s_{3}s_{4}s_{2}\in W(D_{4}) $.}
\end{Example}
The following is \cite[Theorem 5.3]{MR2802176}. 
It can also be obtained from the main result of Dyer \cite{dyer2001rank}.
\begin{Theorem}\label{th:brokenrhombi}
The Schubert variety $ X_{w} $ is rationally smooth if and only if $ [e,w] $ contains no 
broken rhombus.	
\end{Theorem}
\begin{Corollary}\label{co:brokenrhombi}
	Suppose $ W $ is simply laced and let $ w\in W $. 
	Then, the Schubert variety $ X_{w} $ is  smooth if and only if $ [e,w] $ contains no 
	broken rhombus.	
\end{Corollary}

\begin{figure}[htb]
	\centering 
	\begin{tikzpicture}[
	thick, scale=0.7, acteur/.style={circle, fill=black, thick, inner sep=0pt, minimum size=0 cm}
	] 
	\node (a1) at (0,0) {$e$}; 
	\node (a2) at (2.2,4) {2}; 
	\node (a3) at (4.8,4) {4};
	\node (a4) at (-2,4) {3}; 
	\node (a5) at (-4,4) {1};	
	\node (a6) at (4.6,8) {12};
	\node (a7) at (6,8) {42}; 
	\node (a8) at (3,8) {32}; 	
	\node (a9) at (1.8,8) {34};
	\node (a10) at (-0.3,8) {14}; 
	\node (a11) at (-1.3,8) {13}; 
	\node (a12) at (-2.5,8) {24};
	\node (a13) at (-3.9,8) {23}; 
	\node (a14) at (-6,8) {21}; 
	\node (a15) at (7,12) {142};
	\node (a16) at (5,12) {342}; 
	\node (a17) at (3,12) {132}; 	
	\node (a18) at (1.4,12) {212};
	\node (a19) at (0.2,12) {242}; 
	\node (a20) at (-1.1,12) {232};	
	\node (a21) at (-2.7,12) {134};
	\node (a22) at (-4,12) {234};	
	\node (a23) at (-5.5,12) {214}; 
	\node (a24) at (-7,12) {213};
	\node (a26) at (4.7,16) {1342};
	\node (a27) at (2,16) {2142}; 
	\node (a28) at (0,16) {2132};
	\node (a29) at (-2.8,16) {2342};
	\node (a30) at (-5,16) {2134};
	\node (a31) at (0,20) {21342};	
	\draw[->] (a1) -- (a2); 
	\draw[->] (a1) -- (a3); 
	\draw[->] (a1) -- (a4); 
	\draw[->] (a1) -- (a5); 
	\draw[->] (a2) -- (a6); 	
	\draw[->] (a2) -- (a7); 
	\draw[->] (a2) -- (a8);
	\draw[->] (a3) -- (a10); 
	\draw[->] (a4) -- (a11);
	\draw[->] (a2) -- (a12);
	\draw[->] (a2) -- (a13);	
	\draw[->] (a2) -- (a14);	
	\draw[->] (a3) -- (a7);
	\draw[->] (a3) -- (a9); 
	\draw[->] (a3) -- (a12);
	\draw[->] (a4) -- (a8); 
	\draw[->] (a4) -- (a9); 
	\draw[->] (a4) -- (a11);
	\draw[->] (a4) -- (a13); 
	\draw[->] (a5) -- (a6); 
	\draw[->] (a5) -- (a10);
	\draw[->] (a5) -- (a11); 
	\draw[->] (a5) -- (a14); 
	\draw[->] (a6) -- (a15);
	\draw[->] (a6) -- (a17); 
	\draw[->] (a6) -- (a18); 
	\draw[->] (a7) -- (a15);
	\draw[->] (a7) -- (a16); 
	\draw[->] (a7) -- (a19); 
	\draw[->] (a8) -- (a16);
	\draw[->] (a8) -- (a17);
	\draw[->] (a8) -- (a20); 
	\draw[->] (a9) -- (a16);
	\draw[->] (a9) -- (a21); 
	\draw[->] (a9) -- (a22); 
	\draw[->] (a10) -- (a15);
	\draw[->] (a10) -- (a21); 
	\draw[->] (a10) -- (a23); 
	\draw[->] (a11) -- (a17);
	\draw[->] (a11) -- (a21); 
	\draw[->] (a11) -- (a24); 
	\draw[->] (a12) -- (a19);
	\draw[->] (a12) -- (a22); 
	\draw[->] (a12) -- (a23); 
	\draw[->] (a13) -- (a20);
	\draw[->] (a13) -- (a22); 
	\draw[->] (a13) -- (a24); 
	\draw[->] (a14) -- (a18);
	\draw[->] (a14) -- (a23);
	\draw[->] (a14) -- (a24); 
	\draw[->] (a15) -- (a26);
	\draw[->] (a15) -- (a27); 
	\draw[->] (a16) -- (a26); 
	\draw[->] (a16) -- (a29);
	\draw[->] (a17) -- (a26); 
	\draw[->] (a17) -- (a28); 
	\draw[->] (a18) -- (a27);
	\draw[->] (a18) -- (a28); 
	\draw[->] (a19) -- (a27); 
	\draw[->] (a19) -- (a29);
	\draw[->] (a20) -- (a28); 
	\draw[->] (a20) -- (a29); 
	\draw[->] (a21) -- (a26);
	\draw[->] (a21) -- (a30); 
	\draw[->] (a22) -- (a29); 
	\draw[->] (a22) -- (a30);
	\draw[->] (a23) -- (a27); 
	\draw[->] (a23) -- (a30); 
	\draw[->] (a24) -- (a28);
	\draw[->] (a24) -- (a30); 
	\draw[->] (a26) -- (a31); 
	\draw[->] (a27) -- (a31);
	\draw[->] (a28) -- (a31); 
	\draw[->] (a29) -- (a31); 
	\draw[->] (a30) -- (a31);
	\draw[->] (a1) -- (a18); 
	\draw[->] (a1) -- (a19); 
	\draw[->] (a1) -- (a20);
	\end{tikzpicture}
	\caption{The Bruhat graph of $ s_{2}s_{1}s_{3}s_{4}s_{2} \in W(D_{4})$.} 
	\label{fi:D4}  
\end{figure}
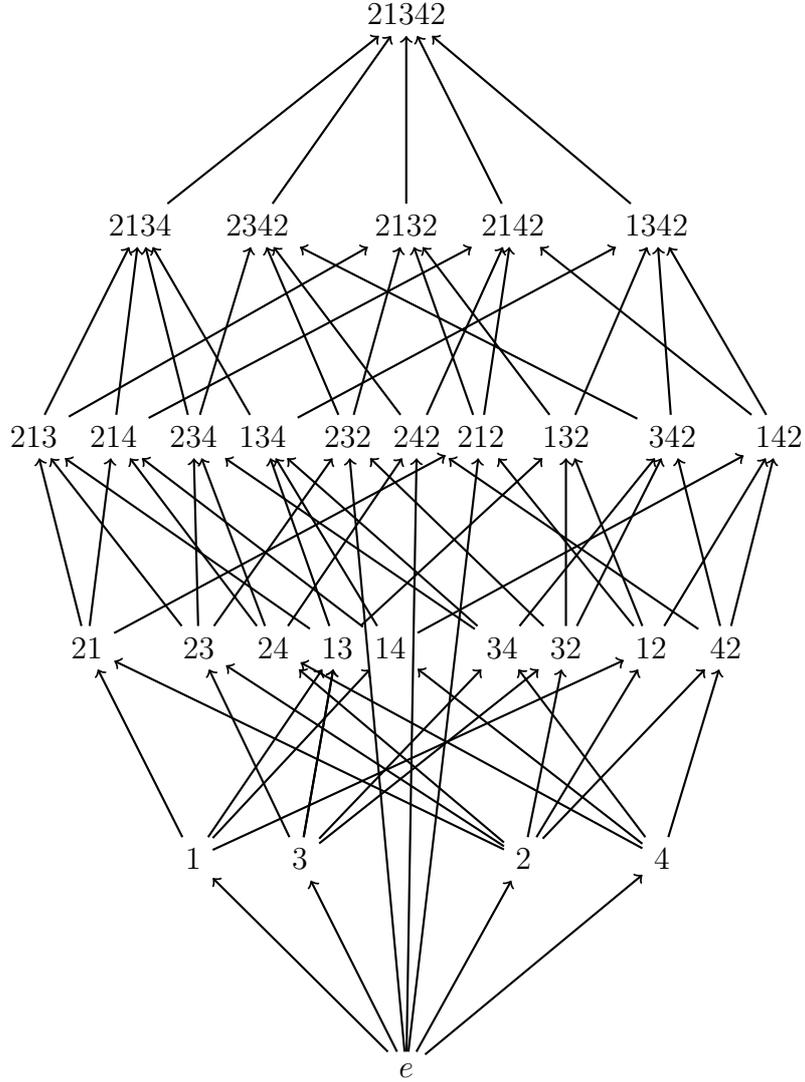

\section{Schubert varieties indexed by involutions} \label{sec:main}

In this section, which contains the main result, we will consider Schubert 
varieties indexed by involutions of finite simply laced groups.

Again let $(W,S)$ be an arbitrary Coxeter system. A {\em parabolic subgroup} of $W$ is a subgroup of the form $W_J = \langle J \rangle$ for $J\subseteq S$. If $ W_J $ is finite its longest element will be denoted by $ w_{0}(J) $.
 
For $ v\in W $ define $ S(v):=\{s \in S: s \leq v\} $. Then $W_{S(v)}$ is the minimal parabolic subgroup of $W$ which contains $v$.
 \begin{Theorem}\label{th:main}
	Suppose $ (W,S) $ is finite and simply laced  
	and let $ v\in W $ be an involution. Then the Schubert variety $ X_{v} $ is 
	smooth if and only if $ v=w_{0}(J) $ for some $ J \subseteq S $.
\end{Theorem}

\begin{proof}
  The ``if'' assertion is obvious: $X_{w_0(J)}$ is a (smooth) flag variety. For the ``only if'' direction, let $ v $ be an involution which is not the longest element of any parabolic
	subgroup $ W_{J} $ of $ W $. Since $v\neq w_{0}(S(v)) $, there 
	exists $ s \in S(v) $ such that $v<sv $. If $\deg_v(e)\neq \ell(v)$, $\Bg{v}{S}$ is not regular and, by Corollary \ref{co:simplylaced}, $ X_{v} $ is not smooth. Thus, we may assume $ \deg_v(e)=\ell (v) $. Since $ \deg_{sv}(e) $ is the number of $ t\in T $ such that $ t\leq sv $ 
	and that degree is at least $ \ell(sv) $,  
	 there exists a reflection 
	$ t \leq sv $  such that $ t\nleq v $. Since $ t\leq sv $, 
	by Lemma \ref{le:invert},  $ t^{-1}\leq v^{-1}s $ which implies that 
	$ t\leq vs $. Also we must have $ st< t $. 
	To see this we use Lemma \ref{le:lifting} in
	the following way: Since $ s\in D_{L}(sv) $, 
	if we would have $ t< st $  
	then by using the lifting property this would imply that $ t\leqslant ssv=v $ 
	which is a contradiction. Now since $ st<t $, by Lemma \ref{le:invert} we see that $ ts<t $.
	Because $ ts< t $  and $ t \leq vs $, we have $ ts< vs $.  
	Since $ s \not\in D_{L}(st)$, we have $ s\in D_{L}(sv)\setminus D_{L}(st) $ and  
	then by Lemma \ref{le:lifting} we get $ st\le v $. 
	Consider the dihedral subgroup $ D =\langle s,sts \rangle$  of $ W $. 
	Since $ W $ is finite and simply laced, then by Lemma \ref{le:simplylaced}, 
	$ D $ is also simply laced. 
	Then $(D,X)$ is either of type $ A_{1}\times A_{1} $ or of type $ A_{2} $, where $X$ is as in Theorem \ref{th:refsubgroup}. 
	By Theorem \ref{th:bgisom} $\Bg{D}{S}$ equals the Bruhat graph of $(D,X)$. Since, moreover, $ st\rightarrow t $, $D$ is not of type $A_{1}\times A_{1}$ and hence $ D $ must be of type $ A_{2} $.
	Therefore, $ D=\left\lbrace  e,s,sts,ts,st,t\right\rbrace $, 
	$ \Bg{D}{S} \cong \Bg{W(A_{2})}{S(A_2)}$, and $ \Bg{D}{S} $ is as shown in Figure \ref{fi:bg(D)}. 
	
	\begin{figure}[H]
		\centering 
		\begin{tikzpicture}[
		thick,
		acteur/.style={
			circle,
			fill=black,
			thick,
			inner sep=0.5pt,
			minimum size=0.2cm
		}
		] 
		\node (a1) at (1,2.25) [acteur,label=$st $]{};
		\node (a2) at (1,0.75)[acteur,label=below:$sts$]{}; 
		\node (a4) at (0,0) [acteur,label=below:$ e $]{}; 
		\node (a5) at (-1,0.75) [acteur,label=below:$s$]{}; 
		\node (a6) at (-1,2.25) [acteur,label=$ts $]{}; 
		\node (a9) at (0,3) [acteur,label=$ t $]{}; 
		\draw[<-] (a1) -- (a2);  
		\draw[->] (a5) -- (a1);
		\draw[<-] (a2) -- (a4);
		\draw[<-] (a5) -- (a4);
		\draw[->] (a6) -- (a9);
		\draw[->] (a1) -- (a9);
		\draw[->] (a5) -- (a6);
		\draw[<-] (a9) -- (a4);
		\draw[<-] (a6) -- (a2);
		\end{tikzpicture}
		\caption{The Bruhat graph of $D$.} \label{fi:bg(D)} 
	\end{figure}
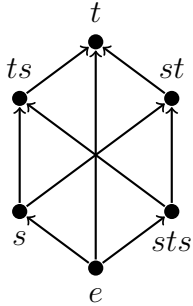
	Now $ ts,st \le v $ but $ t \not \leq v$. From Figure 3,  $ (st,sts,ts) $ is a broken rhombus 
	of $ [e,v] $ because there is no 
	$ x \le v $ such that there are directed edges from $ st $ and $ ts $ 
	to $ x $. To see this, suppose that 
	$ st\rightarrow x\leftarrow ts $. So there exist $ t', t'' \in  T$ with $ t'\neq t''  $ 
	such that $ stt'=tst'' $. Then $ t't''=tsts \neq e$. 
	By Lemma \ref{le:dihedral} we therefore have a dihedral subgroup $ W'= \langle sts,t,t',t'' \rangle$ of $ W $, and $ W' $ is simply laced since $ W $ is (by Lemma \ref{le:simplylaced}). Clearly $ D\subseteq W' $. 
	Since $ W' $ has no more than six elements, $ W'=D $. So $ stt'=x\in D $. 
	Since there is a directed edge 
	from $ st $ to $ x $,  $ x=t \nleq v$.  Since $ (st,sts,ts) $ 
	is a broken rhombus of $ [e,v] $, by Corollary \ref{co:brokenrhombi}, $ X_{v} $ 
	is not smooth.
	\end{proof}

	When $ W $ is not simply laced, there may exist an involution 
	$ w\in W $ which is not the longest element of any parabolic subgroup of $ W $ but for which $ X_{w} $ is smooth. For example, in type $ C_{2} $ there are two involutions of length three. One of them indexes a smooth Schubert variety (and, as was mentioned above, the other one indexes a rationally smooth but not smooth Schubert variety). This example shows that Theorem \ref{th:main} cannot be extended to multiply laced types.  
	We do, however, not know what happens in infinite simply laced types.
	\bibliographystyle{amsplain}
	\bibliography{invschubert}
\end{document}